\documentclass{amsart}
\usepackage[utf8]{inputenc}
\usepackage[utf8]{inputenc}
\usepackage[T1]{fontenc}
\usepackage[all]{xy}
\usepackage{lipsum}
\usepackage{url}
\usepackage{stackrel}
\usepackage{amsmath,amsthm,amssymb}

\usepackage{graphicx}
\newtheorem{theorem}{Theorem}[section]
\newtheorem{lemma}[theorem]{Lemma}
\newtheorem{example}[theorem]{Example}
\newtheorem{proposition}[theorem]{Proposition}
\theoremstyle{definition}
\newtheorem{definition}[theorem]{Definition}

\newtheorem{remark}[theorem]{Remark}
\newtheorem{corollary}[theorem]{Corollary}
\numberwithin{equation}{section}

\begin{document}

%%%%%%%%%%%%%%%%%%%%%%%%%%%%%%%%%%%%%%%%%%%%%%%%%%%%%%%%%%%%
%%%%%%%%%%%%%%%%%%%%%%%%%%%%%%%%%%%%%%%%%%%%%%%%%%%%%%%%%%%%
% This a placeholder for the TOPLOGY PROCEEDINGS logo %%%%%%
%\noindent                                             %%%%%%
%\begin{picture}(150,36)                               %%%%%%
%\put(5,20){\tiny{Submitted to}}                       %%%%%%
%\put(5,7){\textbf{Topology Proceedings}}              %%%%%%
%\put(0,0){\framebox(140,34){}}                        %%%%%%
%\put(2,2){\framebox(136,30){}}                        %%%%%%
%\end{picture}                                        %%%%%%
%%%%%%%%%%%%%%%%%%%%%%%%%%%%%%%%%%%%%%%%%%%%%%%%%%%%%%%%%%%%
%%%%%%%%%%%%%%%%%%%%%%%%%%%%%%%%%%%%%%%%%%%%%%%%%%%%%%%%%%%%
\vspace{0.5in}

\renewcommand{\bf}{\bfseries}
\renewcommand{\sc}{\scshape}
%insert defs/styles
\vspace{0.5in}

\title[Topological Complexity of wedges]%
{Topological Complexity of wedges \\ }

%    Information for first author:
\author{Cesar A. Ipanaque Zapata}
\address{Deparatmento de Matem\'{a}tica,UNIVERSIDADE DE S\~{A}O PAULO
INSTITUTO DE CI\^{E}NCIAS MATEM\'{A}TICAS E DE COMPUTA\c{C}\~{A}O -
USP , Avenida Trabalhador S\~{a}o-carlense, 400 - Centro CEP:
13566-590 - S\~{a}o Carlos - SP, Brasil}
%    Current address (if needed):
%\curraddr{}
\email{cesarzapata@usp.br}
%\thanks{The first author wishes to acknowledge support for this research, from FAPESP 2016/18714-8.}

%    Information for second author (if needed):
%\author{Denise de Mattos}
%\address{Deparatmento de Matem\'{a}tica,UNIVERSIDADE DE S\~{A}O PAULO
%INSTITUTO DE CI\^{E}NCIAS MATEM\'{A}TICAS E DE COMPUTA\c{C}\~{A}O -
%USP , Avenida Trabalhador S\~{a}o-carlense, 400 - Centro CEP:
%13566-590 - S\~{a}o Carlos - SP, Brasil}
%\email{deniseml@icmc.usp.br}

%    General info
%%%%%%%%%%%%%%%%%%%%%%%%%%%%%%%%%%%%%%%%%%%%%%%%%%%
%                                                                                                                           %
%         Please use the current 2010 Mathematics Subject Classification:             %
%         http://www.ams.org/mathscinet/msc/                                                        %
%         http://www.zentralblatt-math.org/msc/en/                                                 %
%%%%%%%%%%%%%%%%%%%%%%%%%%%%%%%%%%%%%%%%%%%%%%%%%%%

\thanks {The first author wishes to acknowledge support for this research, from FAPESP 2016/18714-8.}

\begin{abstract} We prove the formula \begin{equation*}
    TC(X\vee Y)=\max\{TC(X),TC(Y),cat(X\times Y)\}
\end{equation*} for the topological complexity of the wedge $X\vee Y$.
\end{abstract}

\maketitle

%%%%%%%%%%%%%%%%%%%%%%%%%%%%%%%%%%%%%%%%%%%%%%%%%%%%%%%%%%%%%%

\section{\bf Introduction}
Let $X$ be a topological space. We follow a definition of category, one greater than category given in  \cite{cornea2003lusternik}. 

\begin{definition}
We say that the \textit{Lusternik-Schnirelmann category} our category of a topological space $X$, denoted $cat(X)$, is the least integer $n$ such that $X$ can be covered with $n$ open sets, which are all contractible within $X$. 
\end{definition}

Michael Farber, in  \cite{farber2003topological}, defined a numerical invariant $TC(X)$. 

\begin{definition}
The \textit{Topological complexity} of a path-connected space $X$ is the least integer $n$ such that the Cartesian product $X\times X$ can be covered with $n$ open subsets $U_i$, \begin{equation*}
        X \times X = U_1 \cup U_2 \cup\cdots \cup U_n 
    \end{equation*} such that for any $i = 1, 2, \ldots , n$ there exists a continuous function $s_i : U_i \longrightarrow PX$, $\pi\circ s_i = id$ over $U_i$. If no such $n$ exists we will set $TC(X)=\infty$. Where $PX$ denote the space of all continuous paths $\gamma: [0,1] \longrightarrow X$ in $X$ and  $\pi: PX \longrightarrow X \times X$ denotes the
map associating to any path $\gamma\in PX$ the pair of its initial and end points $\pi(\gamma)=(\gamma(0),\gamma(1))$. Equip the path space $PX$ with the compact-open topology. 
\end{definition}

\begin{remark}
 One of the basic properties of $cat(X)$ and $TC(X)$ are its homotopy invariance, respectively (\cite{cornea2003lusternik}, Theorem 1.30; \cite{farber2003topological}, Theorem 3).
\end{remark}

In this paper we present a formula for topological complexity of the wedge $X\vee Y$.

\begin{theorem}\label{tc-wedges-}
Let $X,Y$ be Hausdorff normal topological spaces and  path connected with non-degenerate basepoints, such that $X\times X, Y\times Y$ and $X\times Y$ are normal. Then
\begin{equation}
TC(X\vee Y)=\max\{TC(X),TC(Y), cat(X\times Y)\}.
\end{equation}
\end{theorem}

\begin{remark}\label{poliedro-hausdorff-compacto}
It is known that not only the product of two CW complexes but also the product of a CW complex with any paracompact Hausdorff space is paracompact ( see \cite{morita1963product}, pg. 559). We recall that every paracompact Hausdorff space is normal.  Furthermore, every point of a CW complex is a non-degenerate basepoint (see \cite{rotman1988introduction}, Lemma 8.30, pg. 211).
\end{remark}

We note that, in view of the fact that $cat(X\times Y)\leq cat(X)+cat(Y)-1$, Theorem \ref{tc-wedges-} implies the inequality affirmed by Michael Farber in \cite{farber2006topology}, Theorem 19.1
\begin{equation*}
TC(X\vee Y)\leq\max\{TC(X),TC(Y), cat(X)+cat(Y)-1\}
\end{equation*} under the hypothesis that $X$ and $Y$ are two connected polyhedrons. Moreover, Theorem \ref{tc-wedges-} implies the same equality recently proved by Alexander Dranishnikov in (\cite{dranishnikov2017topological}, Theorem 7)
\begin{equation*}
TC(X\vee Y)=\max\{TC(X),TC(Y), cat(X\times Y)\}
\end{equation*} under the hypothesis that $X$ and $Y$ are two connected CW complexes. Thus, Theorem \ref{tc-wedges-} implies that the condition of dimension considered in (\cite{dranishnikov2017topological}, Theorem 7) can be relaxed.

\begin{example}
We have $TC(\mathbb{RP}^3\vee \mathbb{S}^5)=5$.
\end{example}

\begin{corollary}(\cite{dranishnikov2017topological}, Theorem 2)
For all discrete groups, we have
\[TC(G\ast H)=\max\{TC(G),TC(H),cd(G\times H)+1\}.\] 
\end{corollary}
\begin{proof}
It is well-known the wedge $BG\vee BH$ is a classifying space of the free product $G\ast H$. Also, the product $BG\times BH$ is a classifying space of the product $G\times H$. Furthermore, by (\cite{dranishnikov2017topological}, Theorem 1), for all discrete groups $G$, $cat(G)=cd(G)+1$. Then, by Theorem \ref{tc-wedges-}, we have \begin{eqnarray*} TC(G\ast H) &=& TC(BG\vee BH)\\
            &=& \max\{TC(G),TC(H),cat(BG\times BH)\}\\
            &=& \max\{TC(G),TC(H),cat(B(G\times H))\}\\
            &=& \max\{TC(G),TC(H),cd(G\times H)+1\}.
\end{eqnarray*}            
            
\end{proof}

%%%%%%%%%%%%%%%%%%%%%%%%%%%%%%%%%%%%%%%%%%%%%%%%%%%%%%%%%%%%%%%%%%%%%%%

\section{\bf Preliminaries}

We recall that \begin{equation}
\max\{TC(X),TC(Y), cat(X\times Y)\}\leq TC(X\vee Y)\label{reverse}
\end{equation} where $X,Y$ are any path-connected topological spaces (see \cite{dranishnikov2014topological}, Theorem 3.6).

\begin{proposition}(\cite{cornea2003lusternik}, Lemma 1.25, pg. 13)\label{existencia-cobetura-categorica-baseada}
Let $X$ be a path connected Hausdorff normal space with non-degenerate basepoint $x_0$. If $cat(X)\leq n$, then there is an open categorical cover $\{V_i\}_{i=1}^n$ such that $x_0\in V_i$ for all $i$ and $V_i$ is contractible to $x_0$ relative to $x_0$, that is, there exists a homotopy $H:V_i\times [0,1]\longrightarrow X$ with $H(x,0)=x,\forall x\in V_i$, $H(x,1)=x_0,\forall x\in V_i$ and $H(x_0,t)=x_0,\forall t\in [0,1]$.
\end{proposition}

\begin{lemma}\label{normal-nao-degenerado-secaolocal}
Suppose $X$ is a normal space with non-degenerate basepoint $x_0\in X$. Then, there exits an open neighbourhood of $z_0:=(x_0,x_0)\in X\times X$, $B\subseteq X\times X$, and a local section $s:B\longrightarrow PX$ of $\pi$ such that $s(z_0)(t)=x_0,\forall t\in [0,1]$.
\end{lemma}
\begin{proof}
Note that, by normality and the non-degenerate basepoint hypothesis, there is an open neighbourhood $N$ of $x_0$ and a homotopy $H:N\times [0,1]\longrightarrow X$ with $H(x,0)=x,\forall x\in N$, $H(x,1)=x_0,\forall x\in N$ and $H(x_0,t)=x_0,\forall t\in [0,1]$ (see \cite{aguilar2002algebraic}, Theorem 4.1.16, pg. 94). 

Define an open neighbourhood of $z_0:=(x_0,x_0)$ by $B:= N\times N\subseteq X\times X$. Now, define a homotopy $\widetilde{H}:B\times [0,1]\longrightarrow X\times X$ by \begin{equation}
\widetilde{H}((x,y),t):=(H(x,t),H(y,t)),\forall (x,y)\in B,\forall t\in [0,1].
\end{equation} We note here that
\begin{eqnarray*}
\widetilde{H}((x,y),0) &=& (H(x,0),H(y,0)) \\ &=& (x,y),~\forall (x,y)\in B;\\
\widetilde{H}((x,y),1) &=& (H(x,1),H(y,1)) \\ &=& (x_0,x_0),~\forall (x,y)\in B;\\
\widetilde{H}((x_0,y_0),t) &=& (H(x_0,t),H(x_0,t)) \\ &=& (x_0,y_0),~\forall t\in [0,1].
\end{eqnarray*} Finally, we can define a local section of $\pi$ which satisfies the conclusions of the lemma. Let \[s:B\longrightarrow PX,~ u\mapsto s(u)(t)=\left\{
  \begin{array}{ll}
    p_1\circ \widetilde{H}(u,2t), & \hbox{$0\leq t\leq 1/2$;} \\
   p_2\circ \widetilde{H}(u,2-2t), & \hbox{$1/2\leq t\leq 1$,}
  \end{array}
\right.\] where $p_i:X\times X\longrightarrow X$ for $i=1,2$ denotes projection onto the $i-$th coordinate. 
\end{proof}

\begin{lemma}\label{existencia-algorithmo-planejamento-baseado}
Let $X$ be a path-connected Hausdorff normal space with non-degenerate basepoint $x_0\in X$ and such that the product $X\times X$ is normal. If $TC(X)\leq n$, then there is an open cover $V_1,\ldots,V_n\subseteq X\times X$, such that $z_0:=(x_0,x_0)\in V_i,\forall i=1,\ldots,n$ and over each $V_i$ there exists a continuous function $s_i:V_i\longrightarrow PX$, $\pi\circ s_i = id$ such that $s_i(z_0)(t)=x_0,\forall t\in [0,1]$.
\end{lemma}
\begin{proof}
By $TC(X)\leq n$, let $\{U_i\}_{i=1}^{n}$ be an open cover with respective local sections $\xi_i:U_i\longrightarrow PX$.  Note that, by normality of $X\times X$, there is a refined open cover $\{W_i\}_{i=1}^{n}$ with $W_i\subseteq \overline{W_i}\subseteq U_i$ for each $i=1,\ldots,n$. By Lemma \ref{normal-nao-degenerado-secaolocal}, there exists an open neighbourhood of $z_0$, $B$, and a local section $s:B\longrightarrow PX$ of $\pi$ such that $s(z_0)(t)=x_0,\forall t\in [0,1]$. 

Without loss of generality, we can assume $z_0=(x_0,x_0)\in U_i$ for $i=1,\ldots,k$, for some $1\leq k\leq n-1$ and $z_0=(x_0,x_0)\notin U_i$, for $i=k+1,\ldots,n$. Define a new open neighbourhood of $z_0=(x_0,x_0)$ by \begin{equation}
\mathcal{N} = B\cap U_1\cap\cdots\cap U_k\cap(X\times X-\overline{W}_{k+1})\cap\cdots\cap(X\times X-\overline{W}_{n})\subseteq B.
\end{equation} Note that $z_0\in \mathcal{N}$ and $\mathcal{N}\cap W_j=\emptyset,\forall j=k+1,\ldots,n$.

Now, again by normality of $X\times X$, there exists an open set $M$ with $z_0=(x_0,x_0)\in M\subseteq \overline{M}\subseteq \mathcal{N}$. Note that $\mathcal{N}\subseteq U_i$ for each $i=1,\ldots,k$.

Now we can define the open cover which satisfies the conclusions of the lemma by \begin{equation}
V_i := \left\{
  \begin{array}{ll}
    (U_i\cap (X\times X-\overline{M}))\cup M, & \hbox{ for $i=1,\ldots,k$,} \\
   W_i\cup \mathcal{N}, & \hbox{for $i=k+1,\ldots,n$.}
  \end{array}
\right.
 \end{equation} Note first that $\{V_i\}_{i=1}^{n}$ covers $X\times X$. This fact follows because \[[(U_i\cap (X\times X-\overline{M}))\cup M]\cup (\overline{M}-M)=U_i\] and $\overline{M}\subseteq \mathcal{N}$, implies \[[(U_i\cap (X\times X-\overline{M}))\cup M]\cup \mathcal{N}\supseteq U_i,\] but $W_i\subseteq U_i$, then  \[[(U_i\cap (X\times X-\overline{M}))\cup M]\cup \mathcal{N}\supseteq W_i,\] for each $i=1,\ldots,k$. Then, $\bigcup_{i=1}^nV_i\supseteq \bigcup_{i=1}^n W_i$ but  $\{W_i\}_{i=1}^{n}$ covers $X\times X$, therefore $\{V_i\}_{i=1}^{n}$ covers $X\times X$. Note that $z_0=(x_0,x_0)\in V_i, \forall i=1,\ldots,n$, because $z_0\in M$ and $z_0\in \mathcal{N}$. 
 
 Secondly, each $V_i,~i=1,\ldots,n$ consists of two disjoint (open in $X\times X$) subsets, one subset of $U_i$ not containing the base point and one subset of $\mathcal{N} $ containing the basepoint.  This, allows us to define the following local sections: for $i=1,\ldots,k$, define $s_i:V_i\longrightarrow PX$ by:
 \[ s_i(u) :=\left\{
  \begin{array}{ll}
   \xi_i(u), & \hbox{ if $u\in U_i\cap (X\times X-\overline{M})$;} \\
   s(u), & \hbox{ if $u\in M$.}
  \end{array}
\right.\] Note that $s(z_0)(t)=x_0,\forall t\in [0,1]$ because $z_0\in M$ and $s(z_0)$ is the constant path $C_{x_0}:[0,1]\longrightarrow X,C_{x_0}(t)=x_0,\forall t\in [0,1]$. 

For $i=k+1,\ldots,n$, define  $s_i:V_i\longrightarrow PX$ by:
 \[ s_i(u) :=\left\{
  \begin{array}{ll}
   \xi_i(u), & \hbox{ if $u\in W_i$;} \\
   s(u), & \hbox{ if $u\in \mathcal{N}$.}
  \end{array}
\right.\] Note that $s(z_0)(t)=x_0,\forall t\in [0,1]$ because $z_0\in \mathcal{N}$ and $s(z_0)$ is the constant path $C_{x_0}:[0,1]\longrightarrow X,C_{x_0}(t)=x_0,\forall t\in [0,1]$. 

\end{proof}

We recall the following statement.

\begin{lemma}\label{pasting}
Let $X$ be a topological space with $X=A\cup B$ where $A,B$ are closed subsets in $X$ such that $A\cap B=\{p\}$. Let $U$ be an open subset in $A$ with $p\in U$ and $V$ be an open subset in $B$ with $p\in V$. Then
\begin{enumerate}
    \item[(i)] $U\cup V$ is an open subset in $X$.
    \item[(ii)] If $f:(U,p)\longrightarrow (Z,z_0)$ and $g:(V,p)\longrightarrow (Z,z_0)$ are pointed maps. Then
    $f\cup g:U\cup V\longrightarrow Z$ defined by \begin{equation*} (f\cup g)(x) :=\left\{
  \begin{array}{ll}
    f(x), & \hbox{ if $x\in U$,} \\
   g(x), & \hbox{ if $x\in V$,}
  \end{array}
\right.
\end{equation*} is a continuous map.
    \end{enumerate}
\end{lemma}

%%%%%%%%%%%%%%%%%%%%%%%%%%%%%%%%%%%%%%%%%%%%%%%%%%%%%%%%%%%%%%%%%%%%%%%%%%%%%%%%%%%%%%%%%%%%%%%%%%%%%%%%%%%%%%%%%%%%%%%%%%%%%

\section{\bf Proof of Theorem \ref{tc-wedges-}}

\textit{Proof of Theorem \ref{tc-wedges-}}. 
Note that the product $(X\vee Y)\times (X\vee Y)$ is a union of four closed subsets:\begin{equation}
 (X\times X)\cup  (X\times Y)\cup ( Y\times X)\cup ( Y\times Y),
\end{equation} and any two of these sets intersect at a single point $w_0=((x_0,y_0),(x_0,y_0))$.

Let $n=\max\{TC(X),TC(Y), cat(X\times Y)\}$. Because $TC(X)\leq n$, by Lemma \ref{existencia-algorithmo-planejamento-baseado}, there exists an open cover $\{V_i\}_{i=1}^{n}$ of $X\times X$ such that $w_0\in V_i,\forall i=1,\ldots,n$ and over each $V_i$ there exists a local section $s_i:V_i\longrightarrow P(X)\subseteq P(X\vee Y)$ of $\pi$ such that $s_i(w_0)(t)=(x_0,y_0),\forall t\in [0,1]$. 

\noindent Similarly, because $TC(Y)\leq n$, by Lemma \ref{existencia-algorithmo-planejamento-baseado}, there exists an open cover  $\{U_i\}_{i=1}^{n}$ of $Y\times Y$ such that $w_0\in U_i,\forall i=1,\ldots,n$ and over each $U_i$ there exists a local section $\xi_i:U_i\longrightarrow P(\{x_0\}\times Y)\subseteq P(X\vee Y)$ of $\pi$ such that $\xi_i(w_0)(t)=(x_0,y_0),\forall t\in [0,1]$.

Also, because $cat\left((X\times\{y_0\})\times(\{x_0\}\times Y)\right)=cat(X\times Y)\leq n$, by Proposition \ref{existencia-cobetura-categorica-baseada}, there exists an open categorical cover $\{W_i\}_{i=1}^{n}$ of $X\times Y$ such that $w_0\in W_i,\forall i=1,\ldots,n$ and each $W_i$ is contractible to $w_0$ relative to $w_0$, that is, there exists a homotopy $H:W_i\times [0,1]\longrightarrow (X\times\{y_0\})\times(\{x_0\}\times Y)$ with $H(u,0)=u,\forall u\in W_i$, $H(u,1)=w_0,\forall u\in W_i$ and $H(w_0,t)=w_0,\forall t\in [0,1]$. For each $i=1,\ldots,n$, we can define the following local section $\rho_i:W_i\longrightarrow P(X\vee Y)$ by
\[ \rho_i(u)(t) :=\left\{
  \begin{array}{ll}
    p_1\circ H(u,2t), & \hbox{$0\leq t\leq 1/2$;} \\
   p_2\circ H(u,2-2t), & \hbox{$1/2\leq t\leq 1$.}
  \end{array}
\right.\] where $p_i:(X\times Y)\times (X\times Y)\longrightarrow (X\times Y)$ for $i=1,2$ denotes projection onto the $i-$th coordinate. Note that $\rho_i(w_0)(t)=(x_0,y_0),\forall t \in [0,1]$.

Similarly, because $cat\left((\{x_0\}\times Y)\times(X\times\{y_0\})\right)=cat(Y\times X)\leq n$, by Proposition \ref{existencia-cobetura-categorica-baseada}, there exists an open categorical cover $\{Z_i\}_{i=1}^{n}$ of $Y\times X$ such that $w_0\in Z_i,\forall i=1,\ldots,n$ and each $Z_i$ is contractible to $w_0$ relative to $w_0$, that is, there exists a homotopy $G:Z_i\times [0,1]\longrightarrow (\{x_0\}\times Y)\times(X\times\{y_0\})$ with $G(u,0)=u,\forall u\in Z_i$, $G(u,1)=w_0,\forall u\in Z_i$ and $G(w_0,t)=w_0,\forall t\in [0,1]$. For each $i=1,\ldots,n$, we can define the following local section  $\nu_i:Z_i\longrightarrow P(X\vee Y)$ by
\[ \nu_i(u)(t) :=\left\{
  \begin{array}{ll}
    p_1\circ G(u,2t), & \hbox{$0\leq t\leq 1/2$;} \\
   p_2\circ G(u,2-2t), & \hbox{$1/2\leq t\leq 1$.}
  \end{array}
\right.\] Note that $\nu_i(w_0)(t)=(x_0,y_0),\forall t \in [0,1]$.

Let $T_i:=V_i\cup U_i\cup W_i\cup Z_i$ and note that $T_i$ is an open subset of $(X\vee Y)\times (X\vee Y)$, for each $i=1,\ldots,n$ (see Lemma \ref{pasting}-(i)). Thus $\{T_i\}_{i=1}^{n}$  is an open cover of $(X\vee Y)\times (X\vee Y)$. For each $i=1,\ldots,n$, we can define a local section $\mu_i:T_i\longrightarrow P(X\vee Y)$ of $\pi$ (see Lemma \ref{pasting}-(ii)), by
\[ \mu_i(u)(t) :=\left\{
  \begin{array}{ll}
    s_i(u)(t), & \hbox{ if $u\in V_i$;} \\
    \xi_i(u)(t), & \hbox{ if $u\in U_i$;} \\
    \rho_i(u)(t), & \hbox{ if $u\in W_i$;} \\
   \nu_i(u)(t), & \hbox{ if $u\in Z_i$.}
  \end{array}
\right.\] Because they agree on the intersection. Indeed, $s_i(w_0)=\xi_i(w_0)=\rho_i(w_0)=\nu_i(w_0)=C_{(x_0,y_0)}$ is the constant path $C_{(x_0,y_0)}:[0,1]\longrightarrow X\vee Y,~C_{(x_0,y_0)}(t)=(x_0,y_0),\forall t\in [0,1]$.  Hence, because this open cover has $n$ members, we have \[TC(X\vee Y)\leq \max\{TC(X),TC(Y), cat(X\times Y)\}.\] The reverse inequality was given in (\ref{reverse}). Therefore we have \[TC(X\vee Y)=\max\{TC(X),TC(Y), cat(X\times Y)\}.\]

\begin{flushright}
 $\square$
 \end{flushright}

\bibliographystyle{plain}

\end{document}